\documentclass[12pt,leqno]{amsart}%
\usepackage{amsmath}
\usepackage{amsfonts}
\usepackage{amssymb}
\usepackage{graphicx}%
\providecommand{\U}[1]{\protect\rule{.1in}{.1in}}
\newtheorem{theorem}{Theorem}

\newtheorem{corollary}[theorem]{Corollary}

\newtheorem{definition}[theorem]{Definition}
\newtheorem{example}[theorem]{Example}

\newtheorem{lemma}[theorem]{Lemma}

\newtheorem{proposition}[theorem]{Proposition}
\newtheorem{remark}[theorem]{Remark}

\begin{document}

\title[Spectral and stochastic properties of the $f$-Laplacian]{Spectral and stochastic
properties of the $f$-Laplacian, solutions of PDE's at
infinity\\ and geometric applications}
\author{G. Pacelli  Bessa}
\address{Departamento de Matem\'atica \\Universidade Federal do Cear\'a-UFC\\
60455-760 Fortaleza, CE, Brazil}
\email{bessa@mat.ufc.br}
\author{Stefano Pigola}
\address{Sezione di Matematica - DiSAT\\
Universit\`a dell'Insubria - Como\\
via Valleggio 11\\
I-22100 Como, ITALY}
\email{stefano.pigola{@}uninsubria.it}
\author{Alberto G. Setti}
\address{Sezione di Matematica - DiSAT\\
Universit\`a dell'Insubria - Como\\
via Valleggio 11\\
I-22100 Como, ITALY}
\email{alberto.setti@uninsubria.it}
\begin{abstract}
The aim of this paper is to suggest a  new viewpoint to study qualitative properties
of solutions of semilinear elliptic PDE's defined outside a compact set. The relevant tools
come from spectral theory and from a combination of stochastic properties of the relevant
differential operators. Possible links between spectral and stochastic properties are
analyzed in detail.
\end{abstract}
\maketitle
\tableofcontents

\section*{Introduction}

In this paper we suggest a new perspective to study qualitative properties of solutions
of semilinear elliptic PDE's, especially when these are defined only outside a compact set.
In order to enlarge the range of applicability of the techniques, we decide to place our
treatment in the setting of weighted Riemannian manifolds and corresponding drifted
Laplacians.

 The germ of the present investigation is contained in the very recent paper
\cite{ps-feller} which is devoted to a systematic treatment of the Feller property
of a Riemannian manifold. In fact, using a suitable comparison theory, we shall show
how (weighted) manifolds which are both stochastically complete  and Feller represent
a natural framework where solutions of PDE's at infinity can be studied. The fact that
transience and stochastic completeness of the underlying manifolds have PDE's counterpart
is well understood. However, due to the nature of these stochastic properties, so far
only global solutions have been considered. The introduction of the Feller property
in combination with the stochastic completeness, will enable us to get important
information even in the case of solutions at infinity.

Manifolds which are both
stochastically complete and Feller belong to a pretty wide class containing complete
Ricci solitons, complete manifolds with controlled Ricci tensor and Cartan-Hadamard
manifolds with at most quadratic exponential volume growth. The usefulness of the
technique is visible in the geometric applications which include submanifold theory
and the Yamabe problem. It is well known that the spectral theory of diffusion operators
is affected by stochastic properties of the corresponding diffusion process. For instance,
using capacitary arguments, it is readily seen that recurrence forces  the bottom of the
spectrum of the diffusion operator to be zero. Recent works \cite{BeJoMo, Ha-JGA} have
emphasized an intriguing link between the stochastic completeness of certain specific
manifolds and the essential spectrum of the operator.  We shall analyze in more details
possible relations between stochastic and spectral properties of weighted manifolds.
On the route we will prove a generalized and abstract version of the discreteness of
the spectrum of bounded minimal submanifolds recently obtained in \cite{BeJoMo}. The
nature of the essential spectrum also suggests that it could fit in very well in the
main topic of the paper. Indeed, by the decomposition principle, the bottom of the
essential spectrum is sensitive only on the geometry at infinity of the underlying
manifold and Barta's classical characterization leads naturally to solutions at
infinity of PDE's.
\par
Part of the present work was presented by the third named author at the workshop "Ricci solitons days" held in Pisa, April 4-8, 2011.

\section{Notation}

Throughout this note, we shall always use the symbol $M_{f}$ to denote the
$m$-dimensional weighted manifold%
\[
M_{f}=\left(  M,\left\langle ,\right\rangle ,d\mathrm{vol}_{f}\right)  ,
\]
where $\left(  M,\left\langle ,\right\rangle \right)  $ is a Riemannian
manifold, $f:M\rightarrow\mathbb{R}$ is a selected smooth function on $M$,
$d\mathrm{vol}$ denotes the Riemannian measure of $\left(  M,\left\langle
,\right\rangle \right)  $ and, finally, $d\mathrm{vol}_{f}=e^{-f}%
d\mathrm{vol}$ is the weighted measure. The $f$-Laplacian associated to the
weighted manifold $M_{f}$ is the operator%
\[
\Delta_{f}u=\operatorname{div}_{f}\left(  \nabla u\right)  :=e^{f}%
\operatorname{div}\left(  e^{-f}\nabla u\right)  ,
\]
which is symmetric on $L^{2}\left(  M,d\mathrm{vol}_{f}\right)  $.

\vspace{2mm}
The Bakry-Emery Ricci curvature of the weighted manifold $M_{f}$ is the
$2$-tensor%
\[
Ric_{f}=Ric+\mathrm{Hess}\left(  f\right)  .
\]
In case%
\[
Ric_{f}=\lambda\left\langle ,\right\rangle ,
\]
for some constant $\lambda\in\mathbb{R}$, then the weighted manifold $M_{f}$
is called a Ricci soliton. The Ricci soliton $M_{f}$ is said to be shrinking,
steady or expanding according to the fact that $\lambda>0$, $\lambda=0$ or
$\lambda<0$, respectively.

\section{Stochastic completeness and maximum principle}

We say that the weak maximum principle at infinity holds on a weighted manifold
$M_{f}$ if for every $u\in C^{2}\left(  M\right)  $, with $\sup_{M}u=u^{\ast
}<+\infty$, there exists a sequence $\left\{  x_{k}\right\}  $ along which%
\[
\text{(i) }u\left(  x_{k}\right)  >u^{\ast}-\frac{1}{k},\,\,\,\,\text{ (ii) }%
\Delta_{f}u\left(  x_{k}\right)  <\frac{1}{k}.
\]
It is known, \cite{PRS-Memoirs}, \cite{PRimS-MathZ}, that this principle is
equivalent to the stochastic completeness of the diffusion process associated
to $\Delta_{f}$. This means that the heat kernel of the $f$-Laplacian
$p_{f}\left(  x,y,t\right)  $ satisfies the conservation property%
\[
\int_{M}p_{f}\left(  x,y,t\right)  d\mathrm{vol}_{f}\left(  y\right)
=1\text{,}%
\]
for every $x\in M$ and $t>0$. The $f$-stochastic completeness, in turn, is
implied by the volume condition%
\[
\frac{R}{\log\mathrm{vol}_{f}\left(  B_{R}\right)  }\notin L^{1}\left(
+\infty\right)  .
\]
Accordingly, since for a complete weighted manifold $M_{f}$ satisfying
$Ric_{f}\geq\lambda$ we have, for constants $A$, $B$, $C$, the Qian-Wei-Wylie volume estimate
\cite{Qian-AnnProb}, \cite{WW},%
\begin{equation}
\label{qian_volestimate}\mathrm{vol}_{f}\left(  B_{R}\right)  \leq
A+B\int_{R_{0}}^{R}e^{-\lambda t^{2}+C t}dt, \quad R>>1.
\end{equation}
It follows then that every complete gradient Ricci soliton is $f$-stochastically
complete, therefore, enjoys the weak maximum principle for the $f$-Laplacian.

In fact, using estimates for the potential function due to Z.-H. Zhang,
\cite{Zhang-PAMS}, that will be described in Section \ref{subsection_Feller&Solitons}
 below,  and a general
result contained in \cite{prrs-almostsolitons}, one proves that on every gradient
Ricci soliton, the full Omori-Yau maximum principle holds, both for the $f$-Laplacian and
for the
ordinary Laplacian. Namely,

\begin{theorem}
\label{th_fullomoriyau} Let $(M, \langle\,,\rangle)$ be any gradient Ricci soliton.
Then for every function $u\in C^2(M)$ bounded above, there exists a
sequence $\{x_k\}$  such that  $u(x_n)\to u^*$, $|\nabla u(x_k)|< 1/k$
and $\Delta_f u(x_k)<\displaystyle 1/k$ $($resp. $\Delta u(x_k)<\displaystyle 1/k$$)$.
\end{theorem}

We point out that the case of shrinking solitons was recently obtained by
M. Fern\'{a}ndez-L\'{o}pez and E. Garc\'{\i}a-R\'{\i}o, \cite{FLGR-OmoriYau}.

\section{The Feller property} \label{section_Feller}

It is known, see \cite{PRS-Memoirs}, that the weak maximum principle at infinity
for the $f$-Laplacian is a powerful tool to deduce qualitative information on
the solutions of differential inequalities of the form%
\begin{equation}
\Delta_{f}u\geq\Lambda\left(  u\right)  . \label{ineq}%
\end{equation}
Accordingly, every bounded above solution $u$ of (\ref{ineq}), on the whole
manifold $M$, satisfies%
\[
\Lambda\left(  u^{\ast}\right)  \leq0.
\]
This fact has many applications in geometric analysis. Our aim is now to
investigate qualitative properties of solutions of (\ref{ineq}) which are
defined only in a neighborhood of infinity. This requires the introduction of
new tools that can be developed under the validity of a further stochastic
property of the underlying manifold, namely, the Feller property.

Note that if $M$ is stochastically complete for $\Delta_{f}$, then a bounded
solution $u>0$ of the differential inequality%
\[
\Delta_{f}u\geq\lambda u
\]
outside a smooth domain $\Omega\subset\subset M$ satisfies%
\[
u\left(  x\right)  \leq c\cdot h\left(  x\right)  \text{, on }M\backslash\Omega,
\]
where $c>0$ is a suitable constant and $h>0$ is the minimal solution of the
problem%
\[
\left\{
\begin{array}
[c]{rlll}%
\Delta_{f}h&=&\lambda \cdot h & \text{in }M\backslash\overline{\Omega}\\
h&=&1 & \text{on }\partial\Omega,
\end{array}
\right.
\]
(which is constructed by means of an exhaustion procedure).

Indeed, let $c=\sup_{\partial\Omega} u$. Then, for every $\epsilon>0$,
$$\Delta_{f} (u-c\cdot h -\epsilon)\geq\lambda(u-c\cdot h) \geq\lambda(u-c\cdot h -\epsilon)\,\,{\rm  on}\,\,
M\setminus\overline{\Omega}$$ and $u-c\cdot h - \epsilon\leq-\epsilon$ on
$\partial\Omega$ . Therefore the function $v_{\epsilon}=\max\{0,
u-c\cdot h-\epsilon\}$ is bounded, non-negative and satisfies $\Delta_{f}
v_{\epsilon}\geq\lambda v_{\epsilon}$. Since $M$ is stochastically complete
with respect to $\Delta_{f}$, $v_{\epsilon}\equiv0$, that is, $u\leq ch
+\epsilon$, and the conclusion follows letting $\epsilon\to0.$

In particular, if $h\left(  x\right)  \rightarrow0$ as $x\rightarrow\infty$,
we can deduce that the same holds for the original function $u$.

 According to
a  characterization by R. Azencott \cite{Az-bsmf}, it happens that the required
decay property of $h$ is equivalent to  the Feller property on
$M$ with respect to $\Delta_{f}$, that is, that the heat semigroup generated by $-\Delta_{f}$
maps the space  $C_{o}(M)$ of continuous functions vanishing at infinity into itself, or equivalently,
that for every relatively
compact open set $\Omega$ in $M$, the heat kernel $p_f$ of $-\Delta_f$
satisfies
\[
\int_{\Omega} p_f(x,y,t) d\mathrm{vol}_{f}(y)\to0 \text{ as } x\to+\infty.
\]

We thus obtain the following result.
%
%
%
\begin{theorem}\label{thm_stoc+feller}
Let $M_{f}$ be \thinspace$f$-stochastically complete. If $M_{f}$ is Feller,
then every bounded solution $v>0$ of%
\[
\Delta_{f}v\geq\lambda v\text{, on }M\backslash\overline{\Omega}
\]
satisfies%
\[
v\left(  x\right)  \rightarrow0\text{, as }x\rightarrow\infty.
\]
\end{theorem}
On the basis of these observations, we prove the following theorem.

\begin{theorem}
\label{th_f(u)}Let $M_{f}$ be a stochastically complete and Feller manifold
for $\Delta_{f}$. Consider the differential inequality%
\begin{equation}
\Delta_{f}u\geq\Lambda\left(  u\right)  ,\text{ on }M\backslash\Omega,
\label{Dirichlet}%
\end{equation}
where $\Omega\subset\subset M$ and $\Lambda:[0,+\infty)\rightarrow
\lbrack0,+\infty)$ is either continuous or it is non-decreasing function which
satisfies the following conditions:%
\[
\text{(a) }\Lambda\left(  0\right)  =0\text{; (b) }\Lambda\left(  t\right)
>0\text{, }\forall t>0\text{; (c) }\liminf_{t\rightarrow0+}\frac
{\Lambda\left(  t\right)  }{t^{\xi}}>0,
\]
for some $0\leq\xi\leq1$. Then, every bounded solution $u>0$ of
(\ref{Dirichlet}) satisfies%
\[
\lim_{x\rightarrow\infty}u\left(  x\right)  =0.
\]
\end{theorem}

\begin{proof}
Let us consider the case where $f$ is not continuous, the other case is
easier. By assumption, there exists $0<\varepsilon<1/2$ and $c>0$ such that%
\[
\Lambda\left(  t\right)  \geq c\, t^{\xi},\text{ on }(0,2\varepsilon).
\]
Since%
\[
t^{\xi}\geq t\text{, on }(0,1],
\]
and $\Lambda$ is non-decreasing, then%
\[
\Lambda\left(  u\left(  x\right)  \right)  \geq\Lambda_{\varepsilon}\left(
u\left(  x\right)  \right)  =\left\{
\begin{array}
[c]{cc}%
c\, u, & \text{if }u\left(  x\right)  <\varepsilon\\
c\, \varepsilon, & \text{if }u\left(  x\right)  \geq\varepsilon.
\end{array}
\right.
\]
On the other hand, since $u>0$ is bounded, if we set $u^{\ast}=\sup
_{M\backslash\Omega}u$, then%
\[
c\varepsilon\geq\frac{c\,\varepsilon}{u^{\ast}}u^{\ast}\geq\frac{c\,\varepsilon
}{u^{\ast}}u.
\]
It follows that%
\[
\Delta_{f}u\geq\Lambda_{\varepsilon}\left(  u\right)  \geq\lambda u,
\]
where%
\[
\lambda=c\min\left\{  1,\frac{\varepsilon}{u^{\ast}}\right\}  >0.
\]
Using the Feller property we now conclude that $u\left(  x\right)
\rightarrow0$, as $x\rightarrow\infty$.
\end{proof}


\section{Estimates for the f-Laplacian of the distance function and Comparison results}
Because of comparison arguments and radialization techniques,
many of the properties of solutions of differential in(equalities) involving
the  $f$-Laplacian, and in particular, the stochastic properties of $\Delta_f$
may be deduced imposing suitable bounds on $\Delta_f r$ where $r(x)$ denotes the
distance function from a reference point $o\in M$. We are going to collect some results along
these lines concerning stochastic completeness, the Feller property, the full Omori-Yau maximum
principle and the compact support principle. Items (i) and (ii) in the next theorem are weighted versions of Corollary 15.2 (c) and (d) in \cite{grigoryan-BAMS}.

\begin{theorem}
\label{th-comparison1}
Let $M_f$ be a weighted Riemannian manifold, and let $r(x)$ be the distance function
from a fixed point $o\in M$. Let also $g$ be a $C^2$ odd function on $\mathbb{R}$ satisfying
$g(0)=0$, $g'(0)=1$ and $g(t)>0$ for all $t>0.$
\newline
(i) Assume that there exist $R_0\geq 0$ and constant $n>1$ such that
for every $x$ within the cut locus of $o$  with $r(x)>R_0$ we have
\begin{equation}
\label{Delta-r-upperbound}
\Delta_f r(x) \leq (n-1) \frac{g'}{g}(r(x)) \, \text{ with } \,
 \frac{\int_0^r g(t)^{n-1} dt}{g(r)^{n-1}}\not\in L^1(+\infty),
\end{equation}
then $M_f$ is stochastically complete.
\vspace{1mm}
\newline (ii) Assume that $o$ is a pole and that there exist $R_0\geq 0$ and a constant $n>1$ such
that,
for $r(x)>R_0$, we have
\begin{equation}
\label{Delta-r-lowerbound}
\Delta_f r(x) \geq (n-1) \frac{g'}{g}(r(x)) \, \text{ with } \,
 \frac{\int_0^r g(t)^{n-1} dt}{g((r)^{n-1}}\in L^1(+\infty),
\end{equation}
then $M_f$ is not stochastically complete.
\vspace{1mm}
\newline
(iii) Assume that $o$ is a pole,  that there exist $R_0\geq 0$ and a constant $n>1$ such that for
$r(x)>R_0$ we have
\begin{equation}
\label{Delta-r-lowerbound1}
\Delta_f r(x) \geq (n-1) \frac{g'}{g}(r(x))
\end{equation}
and that
either%
\begin{equation}
\frac{1}{g^{n-1}\left(  r\right)  }\in L^{1}\left(  +\infty\right)
\label{model1}%
\end{equation}
or%
\begin{equation}
\text{(a)}\;\frac{1}{g^{n-1}\left(  r\right)  }\notin L^{1}\left(
+\infty\right)  \text{\qquad and\qquad(b) }\;\frac{\int_{r}^{+\infty}%
g^{n-1}\left(  t\right)  dt}{g^{n-1}\left(  r\right)  }\notin L^{1}\left(
+\infty\right), \label{model2}%
\end{equation}
then $M_f$ is Feller.
\vspace{1mm}
\newline
(iv) Assume there exist $R_0\geq 0$ and a constant $n>1$ such that if $x$ is within the cut locus of $o$
with  $r(x)>R_0$ we have
\begin{equation}
\label{Delta-r-upperbound1}
\Delta_f r(x) \leq (n-1) \frac{g'}{g}(r(x))
\end{equation}
and that $g$ does not satisfy the conditions in (iii). Then
$M_f$ is not Feller.
\end{theorem}
\begin{proof}
We outline the proof, which follows the lines of arguments valid for
the ordinary Laplacian. To prove (i) and (ii), let $\alpha(r)$ be the
function defined by
\[
\alpha(r) = \int_0^r  \frac{\int_0^t g(s)^{n-1} ds}{g((t)^{n-1}}\,
dt
\]
and note that
\[
\alpha'(r)>0\qquad
\alpha'' (r) + (n-1)\frac{g'}g \alpha'(r) =1.
\]
Let also $v(x)=\alpha(r(x))$, so that $v$ is $C^2$ within the cut locus of $o$ and there
\[
\Delta_f v(x) = \alpha''(r(x)) + \Delta_f r(x) \alpha'(r(x)).
\]
Now, in case (ii) $v$ is $C^2$ and bounded above on $M$, and since $\alpha'>0$
it does not attain a maximum. Thus, if $R$ is large enough that
\eqref{Delta-r-lowerbound} holds in $B_R^c$ and $\gamma<u^*$ is sufficiently
 close to $u^*$ that $\Omega_\gamma =\{x: u(x)> \gamma \}\subset B_R^c$, then,
 in $\Omega_\gamma$
\[
\Delta_f v \geq \alpha''(r(x)) + (n-1)\frac{g'} g \alpha'(r(x)) =1,
\]
so that $v$ violates the weak maximum principle at infinity, and $M_f$
is not stochastically complete.

To prove (i), let $u$ be $C^2$ and bounded above on $M$. We claim
that for every $\gamma< u^*$, $\inf_{\Omega_\gamma} \Delta_f u\leq 0$ which
clearly implies that the weak maximum principle at infinity holds for $\Delta_f$
and $M_f$ is stochastically complete.

Note that in the present situation $v$ tends to infinity as $r(x)\to \infty$
and satisfies
\[
\Delta_f v \leq \alpha'' + (n-1) \frac{g'} g \alpha' \leq 1
\]
for $x$ within the cut locus of $o$ and  such that $r(x)>R_0$.

Assume by contradiction that there exists $\gamma<u^*$ such that
$$\inf_{\Omega_c}\Delta_f u\geq 2c>0.$$
Clearly $u$ does not attain its supremum,
and, by taking $\gamma$ close enough to $u^*$, we may arrange that
\eqref{Delta-r-upperbound} holds on $\Omega_{\gamma}$. Let
 $x_o\in \Omega_\gamma$ and choose $0<\delta<c$  small enough that the function
$\tilde u=u-\gamma -\delta v$ is positive at $x_o$. Since  $\tilde u
<0$ on $\Omega_\gamma$, and tends to $-\infty$ as $r(x)\to +\infty$, it attains
a positive maximum at $\overline x \in \Omega_\gamma$, and using the Calabi trick
we may assume that $r(x)$ be smooth at $\bar x$.   Then, at $\bar x $,
\[
\Delta_f \tilde u(\bar x) =
\Delta _f u (\bar x) - \delta \Delta_f v (\bar x)\geq 2c-\delta \geq c>0,
\]
which yields the required contradiction.

We now come to the Feller property. In case (iii), the conditions satisfied by
$g$ imply that the model manifold $M_{g}^{n}$  defined as $\mathbb{R}^{n}$ endowed
with the metric%
\[
\left\langle ,\right\rangle =dr^{2}+g\left(  r\right)  ^{2}d\theta^{2},
\]
is Feller (with respect to the ordinary Laplacian), and,
according to \cite{ps-feller} Theorem~4.4 and Lemma~5.1, the radial minimal
solution $\beta$ of the exterior boundary value problem
\begin{equation}
\label{ext_problem_model}
\begin{cases}
\beta'' + (n-1) {\displaystyle \frac{g'}{g}} \beta' = \lambda \beta&\\
\beta(R_0) =1 &
\end{cases}
\end{equation}
tends to $0$ as $r$ tends to infinity and satisfies $h'(r)<0.$
As above, define $v(x)= \beta(r(x)$, so that $v(x)\to 0$ as $r(x)\to \infty$ and
\[
\Delta_f v = \beta'' + \Delta_f r \beta' \leq \beta'' + (n-1) \frac{g'}g \beta'
=\lambda v.
\]
For every $R>R_0$, let $h_R$ be the solution of the exterior boundary value problem
\begin{equation*}
\begin{cases}
\Delta_f h_R= \lambda h_R&\\
h_R(R_0) =1,\,\, h_R(R)= 1, &
\end{cases}
\end{equation*}
and note that by the comparison principle $h_R <v$ in $B_R\setminus B_{R_0}.$
As $R\to +\infty,$  $h_R$ tends to the minimal solution $h$ of the problem
\begin{equation}
\label{ext-problem}
\begin{cases}
\Delta_f h= \lambda h&\\
h(R_0) =1. &
\end{cases}
\end{equation}
Since clearly $0<h<v$ in $M\setminus B_{R_0}$, $h$ tends to $0$ as $r(x)$
tends to $\infty$ and $M_f$ is Feller.

Finally, assume that \eqref{Delta-r-upperbound1} holds and that $g$
does not satisfy the conditions in (iii) hold.  Note that in particular,
$g^{-n+1}\not\in L^1(+\infty)$ while  $g^{n-1}\in L^1(+\infty)$, and  according to i)
$M_f$ is stochastically complete. By \cite{ps-feller} Theorem~4.4,
the model manifold $M^n_g$ is not Feller, and therefore the minimal radial solution
$\beta$ of the exterior problem \eqref{ext_problem_model} does not tend to zero
as $r\to+\infty$. Since $\beta'\leq 1$ by \cite{ps-feller} Corollary 5.1,
it follows that the function $v(x)=\beta(r(x))$ is a bounded solution of
\[
\Delta_f v\geq \lambda v
\]
which satisfies $v=1$ on $R_0$ and which does not tend to zero at infinity.
 It follows from the discussion preceding Theorem~\ref{thm_stoc+feller} that
the minimal solution of the exterior problem  \eqref{ext-problem} satisfies
$v\leq c \cdot h$ for some constant $c>0$, so that $h$ does not tend to zero at
infinity and $M_f$ is not Feller.
\end{proof}

\begin{remark}
{\rm
In the case that $n$ is an integer, the differential inequalities satisfied by $\Delta_{f}r$ can be interpreted as comparison with the Laplacian of the distance function of model $M_{g}^{n}$.
It is also interesting to observe that in the case of the Feller property
the inequalities assumed for $\Delta_f r$ go in the opposite direction than those
assumed in the case of stochastic completeness.
}
\end{remark}


\
\smallskip

In the case of the ordinary Laplacian, upper and lower estimates for
$\Delta r$ may be obtained via the Laplacian comparison theorem imposing
lower bounds on the Ricci curvature, or upper bounds on the sectional curvature,
respectively. In the case of the $f$-Laplacian, there does not seem to be an
analogue of the sectional curvature whose control allows to obtain lower
estimates for $\Delta_f r$. As for upper estimates, the most effective way
to obtain upper bounds for $\Delta_f r$ is to impose lower bounds on the
modified Bakry-Emery Ricci tensor
\[
Ric_{\alpha}= Ric +\mathop{Hess} f -\frac 1 \alpha df\otimes df
\]
with $\alpha>0,$ but in view of applications to Ricci solitons it
is important to try and obtain estimates for $\Delta_f r$ assuming
lower bound on the Bakry-Emery Ricci tensor $\mathop{Ric}_f$ corresponding
to $\alpha= +\infty$, together with some control on the weight $f$ and
or its gradient.

Indeed, it was shown by Qian, \cite{Qian-AnnProb} Theorem 2.1, in the more general
case of operators of the form $\Delta + X$ with a drift which is
not necessarily a gradient, that upper estimates for $\Delta_f r$
follow from imposing lower bounds on $\mathrm{Ric}_f$ and a control
on the drift term $X$.
More precisely we have:

\begin{theorem}
\label{qian-comparison}
Let $M_f$ be a weighted manifold, let $o$ be a reference point in $M$ and let  $r(x)=d(x,0)$
be the Riemannian distance function from $o$.
\newline
(i)  Assume that
\[
\mathrm{Ric}_f\geq - k^2
\]
for some constant $k\geq 0$, then there exists a constant $C$ depending only on
$M$, on $\nabla f$ and on $o$ such that
\[
\Delta_f r(x)\leq C + \frac{m-1} {r(x)}  + k^2\cdot r(x) \,\, \text{ on
 }\,\, M\setminus \mathrm{cut}(o).
\]
(ii) Assume that
\begin{equation*}
\begin{cases}
\mathrm{Ric}_f (x) \geq - k_1^2 (r(x)) &\\
|\nabla f| (x)\leq k_2 (r(x)), &
\end{cases}
\end{equation*}
where $k_i(r)$ are continuous non-decreasing functions satisfying
$k_i(r)\to +\infty$ as $r\to +\infty$. Then
\[
\Delta_f r(x) \leq m \frac {g'(r(x))}{g(r(x)}\,\, \text{ on
 }\,\, M\setminus \mathrm{cut}(o),
\]
where $g:[0,+\infty )\to [0,+\infty)$ is the solution of the initial
value problem
\begin{equation}
\label{g_ivp}
\begin{cases}
g''(r) - \displaystyle{\frac{k_1(r)^2+k_2(r)^2} m} g(r) = 0&\\
g(0)=0,\,\, g'(0)=1.&
\end{cases}
\end{equation}
(iii) Assume that there are nonnegative constants $k$ and $C$ such that
\begin{equation*}
\begin{cases}
\mathrm{Ric}_f \geq - k &\\
|\nabla f| \leq C(d(x,o) +1) \quad \forall x\in M.&
\end{cases}
\end{equation*}
Then for every $p\in M$, and every $x\in M\setminus
\left(\{p\}\cup \mathrm{cut}(p)\right)$, if $\rho(x)=d (x,p)$, then
\[
\Delta_f \rho(x) \leq \frac{m-1} \rho +\frac 13 (k+2C) \rho +
C(1+d(o,p)).
\]
\qed
\end{theorem}

Applying Theorem~\ref{th-comparison1} we obtain a versions for the
$\Delta_f$ of results by Qian,  \cite{Qian-AnnProb}, Theorems 1.5 and 1.6, on
the stochastic completeness of Laplacians with drift. We point out
that the proof we present is entirely deterministic.

\begin{theorem}
\label{stoch_completeness}
Let $M_f$ be a complete weighted Riemannian manifold, and assume
that either
\begin{equation}
\label{stoch-completeness1}
\mathrm{Ric}_f\geq - k^2
\quad \text{ for some constant }\,\, k\geq 0,
\end{equation}
or
\begin{equation}
\label{stoch-completeness2}
\begin{cases}
\mathrm{Ric}_f (x) \geq - k_1^2 (r(x)) &\\
|\nabla f| (x)\leq k_2 (r(x)), &
\end{cases}
\end{equation}
where $k_i(r)$ are continuous non-decreasing functions satisfying
$k_i(r)\to +\infty$ as $r\to +\infty$ and
\begin{equation}
\label{stoch-completeness2a}
\frac 1{\sqrt{k_1^2(t)+k_2^2(t)}}\not\in L^1(+\infty).
\end{equation}
Then $M_f$ is stochastically complete for $\Delta_f$.
\end{theorem}
\begin{proof}
If \eqref{stoch-completeness1} holds, then $\Delta_f r$
satisfies the estimate (i) in Theorem~\ref{qian-comparison},
and it is easily seen that this implies \eqref{Delta-r-upperbound} with
$n=m$ and $$g(r)=r\exp\bigl[\frac 1{m-1}\bigl(Cr + \frac k2
r^2\bigr)\bigr],$$ and it is clear that $g(r)$ satisfies the
non-integrability in \eqref{Delta-r-upperbound}, and $M_f$ is stochastically complete by
Theorem~\ref{th-comparison1} (i).

Assuming  that \eqref{stoch-completeness2} and
\eqref{stoch-completeness2a} hold, then $\Delta_f r(x)$ satisfies the
estimate in Theorem~\ref{qian-comparison} (ii) with $g$  solution of
\eqref{g_ivp}. This implies that $\psi = g'/g$ satisfies the Riccati
equation
\[
\psi' + \psi^2 = k^2,\quad \text{ with } k(r)^2 =
\frac{k_1(r)^2+k_2(r)^2}{m},
\]
and  $\psi(r) = \displaystyle{\frac 1 r} + O(1)$ as $r\to 0$. Thus $\psi'<0$ whenever $\psi>k$ and since
$k(r)$ is increasing and tends to $+\infty$ as $r\to +\infty$,
standard arguments show that there exists $r_o$ such that $\psi(r_o)= k(r_o)$ and
$\psi(r)\leq k(r)$ for all $r\geq r_o$.

This in turn implies that there exists  $R$ sufficiently large and a
constant $C>0$ such that
\[
\frac{\int_0^r g(t)^m dt }{g(r)^m} \geq  \frac C
{ k(r)} \quad \text{if }\, r\geq R.
\]
Indeed,  $g(r)\to +\infty $ as $r\to +\infty$ and,  using de l'Hospital's
rule, and the fact that $k$ is non-decreasing we have
\[
\liminf_{r\to \infty} \frac{k(r) \int_0^r g(t)^m dt}{g(r)^m} \geq
\liminf_{r\to \infty}\frac{ k(r) g(r)}{m g'(r)} \geq \frac 1 m.
\]
By \eqref{stoch-completeness2a}, $1/k(r)\not\in L^1(+\infty)$ and, again  by
Theorem~\ref{th-comparison1} (i), $M_f$ is stochastically complete.
\end{proof}

We stress that, in order to deduce the validity of the Feller property
using Theorem~\ref{th-comparison1} (iii), one needs a control from
below on $\Delta_f r$, while lower bounds on $Ric_f$ typically produce upper
estimates on $\Delta_f r$.

Using a probabilistic technique which extends to the Laplacian with a drift term
a previous result of P. Hsu, \cite{Hsu-AnnProb} for $\Delta$, and which essentially
consists in genuine estimates on hitting time probabilities, Qian \cite[Theorem  1.7]{Qian-AnnProb}
 proves the following theorem. Note that, again,  Qian's result actually applies
to the more general operators of the form $L=\Delta +X$.

\begin{theorem}
\label{qian-feller}
Let $M_f$ be a complete weighted manifold and assume that for some $o\in M$ the
conclusion of Theorem~\ref{qian-comparison} (iii) holds, namely, there exist constants
$k$ and $C>$ such that,  for every $p\in M$, if $\rho(x)=d(x,p)$ denotes the distance function
from $p$, we have
\[
\Delta_f \rho (x)
 \leq \frac{m-1} \rho +\frac 13 (k+2C) \rho +
C(1+d(o,p)).
\]
Then $M_f$ is Feller with respect to $\Delta_f$. In particular, if
$Ric_f$ and $\nabla f$ satisfy the estimates in
Theorem~\ref{qian-comparison} (iii), then $M_f$ is Feller with
respect to $\Delta_f.$
\end{theorem}


We next describe a result which extends a condition on the validity of
the full Omori-Yau maximum principle for the operator $\Delta_f$
proved in \cite{prrs-almostsolitons}. The argument we are going to
use is an  adaptation of a recent elegant proof of the Omori-Yau maximum principle
due to A. Borb{\'e}ly, \cite{borbely-bullaustr} and \cite{borbely-omori}.
We are grateful to A. Borb{\'e}ly for sending us a copy of
\cite{borbely-omori}.

\begin{theorem}
\label{generalymoriyau}
Let $M_f$ be a complete weighted  manifold and assume that
there exists a non-negative $C^2$ function $\gamma$ satisfying the following conditions
\begin{eqnarray}
\label{hp1} &\gamma\left(x\right)\rightarrow +\infty
\text{ as } x\rightarrow\infty\\
\label{hp2} &\exists A>0 \text{ such that } \left|\nabla \gamma\right|\leq
A\text{ off a compact set}\\
\label{hp3}&\exists B>0 \text{ such that } \Delta_{f}\gamma\leq B G\left(\gamma \right)
 \text{ off a compact set}
\end{eqnarray}
where $G$ is a smooth function on $\left[0,+\infty\right)$ satisfying
\begin{equation}\label{hp5}
\left(i\right)\,G\left(0\right)>0 \quad\left(ii\right)\,G^{\prime}\left(t\right)\geq 0
\text{ on } \left[0,+\infty\right)\quad \left(iii\right)G\left(t\right)^{-1}\notin
L^{1}\left(+\infty\right).
\end{equation}
Then the Omori-Yau maximum principle for $\Delta_{f}$ holds. The
same conclusion holds if $\gamma (r)=r(x)$ is the distance function
from a reference point $o$ and we assume that inequality
\eqref{hp3} holds in the complement of the cut locus of $o$ (while of course
\eqref{hp1} and \eqref{hp2} are automatically satisfied).
\end{theorem}
\begin{proof}We outline the proof which follows closely Borbely's argument.

Let $u$ be a $C^2$  function such that $u^* = \sup_M
u<+\infty$. We want to show that there exists a sequence $\{x_n\}$ such that
$u(x_n)\to u^*$, $\lim_n |\nabla u(x_n)|= 0$ and $\limsup_n\Delta_f u(x_n) \leq
0$.
We may assume that $u$ does not attain its supremum for otherwise the
conclusion is obvious. Define
\[
F(t) = \exp\{\int_0^t G(s)^{-1} ds\},
\]
so that $F$ is non-decreasing and $\lim_{t\to +\infty}F(t) =
+\infty$.  For every $t$, let  $\Omega_t=\{x \, :\, \gamma(x)>t\}$. Since $\gamma$
is proper, there exists $t_o$  such that inequalities \eqref{hp2} and \eqref{hp3} hold
in the set $\Omega_{t_o}$. Let $\epsilon<\min\{1, u^*-\sup_{\Omega_{t_o}^c} u
\}$, and define  $h_\lambda (x)= \lambda F(\gamma(x)) + u^*-\epsilon$.
Since $F\geq 1,$ if $\lambda >\epsilon$ then $h_\lambda >u^*>u$ on
$M.$ Let $\lambda_o=\inf\{\lambda \,:\, h_\lambda(x)>u(x) \,\forall x\in
M\}$, and note that since $u<u^*$ on $M$ then $\lambda_o>0,$ and, by
continuity, $h_{\lambda_o}(x)\geq u(x)$ for every $x$ in $M$.

We claim that there exists $x_\epsilon $ such that $h_{\lambda_o}(x_\epsilon)=
u(x_\epsilon)$.
Note that since $h_{\lambda_o}> u^*-\epsilon>\sup_{\Omega_{t_o}^c}
u$, $x_\epsilon$ lies necessarily in $\Omega_{t_o}$.

To prove the claim, we will show that if  $h_{\lambda}>u$ on $M$ for some $\lambda>0$ then
there exists $\lambda'<\lambda$ such that $h_{\lambda'}>u$ on $M$.
Indeed, since $F(\gamma)\to +\infty$ as $\gamma\to +\infty$,  and
$\gamma$ is proper, there exists $t_1>t_o$ such  that $h_\lambda
>u^*+1$ in $\Omega_{t_1}$. Since $\Omega_{t_1}^c$ is compact, and $h_\lambda(x)>u(x)$
on $M$, we may choose  $\lambda'<\lambda$  sufficiently close to
$\lambda$ so that $h_{\lambda'}>u$ on $\Omega_{t_1}^c$, and
$h_{\lambda'}>u^*$ on $\partial \Omega_{t_1}$. Since $F$ is increasing, $h_{\lambda'}>u^*$ in
$\Omega_{t_1}$ and therefore $h_{\lambda'}>u$ on $M$, as required.

Next we claim that $h_{\lambda_o}$ is smooth at $x_\epsilon$. This
is clear if $\gamma$ is $C^2$ on $M$, while if $\gamma(x)=r(x)$ is
the Riemannian distance function, the proof in Borbely's paper,
which only uses the fact that $u- h_{\lambda_o}$ attains a maximum at $x_\epsilon$,
and properties of the function $r(x)$ applies without changes.

Thus, since
$u(x_\epsilon)=h_{\lambda_o}(x_\epsilon)=
\lambda_o F(\gamma(x_\epsilon)) + u^* - \epsilon< u^*= \sup u
$
we have
\[
u(x_\epsilon) >u^* - \epsilon \,\text{ and } \,
\lambda_oF(\gamma(x_\epsilon)) <\epsilon.
\]
Also, since $u-h_{\lambda_o}$ attains a maximum at $x_\epsilon$,
\begin{equation*}
\begin{split}
& (i)\,\, \nabla u(x_\epsilon) = \lambda_o F'(\gamma(x_\epsilon))
\nabla \gamma (x_\epsilon) \\
& (ii) \, \Delta_f u(x_\epsilon) \leq \Delta_f h_{\lambda_o}
(x_\epsilon)= \lambda_o [F''(\gamma(x_\epsilon)) |\nabla
\gamma(x_\epsilon)|^2 + F'(\gamma(x_\epsilon)) \Delta_f
\gamma(x_\epsilon)].
\end{split}
\end{equation*}
Easy computations show that $F'=F/G$ and $F''\leq F/G^2$ so that (i)
above and \eqref{hp2} yield
$$
|\nabla u(x_\epsilon)|\leq  A\lambda_o
\frac{F(\gamma(x_\epsilon))}{G(\gamma(x_\epsilon))} < \frac{A}{G(0)}
\epsilon,
$$
while using (ii) and \eqref{hp3} we get
\[
\Delta_f u(x_\epsilon) \leq \lambda_o F(\gamma(x_\epsilon))
[\frac 1{G^2(\gamma(x_\epsilon))} + \frac{\Delta_f
u(x_\epsilon)}{G(\gamma(x_\epsilon))}]\leq
(1+B)\lambda_oF(x_\epsilon)<(1+B) \epsilon.
\]
\end{proof}

We conclude this section with a brief discussion of the compact support
principle for the operator $\Delta_f$ (see, \cite{PSZ-JMPA} \cite{puriser-jde},
where more general quasilinear elliptic operators in  divergence form are considered).

A function $u$ is said to be a semiclassical solution of the differential inequality
\begin{equation}
\label{ineq-cpt-supp}
\Delta_f u \geq \lambda(u)
\end{equation}
in a domain $\Omega$ if $u\in C^1(\Omega)$ and $u$ satisfies the
inequality in weak sense, that is
\[
\int_\Omega [\langle \nabla f,\nabla \phi\rangle + \lambda (u)\phi]\leq 0
\]
for every non-negative $\phi\in C^1_c(\Omega)$.

The compact support principle is said to hold for the differential inequality
\eqref{ineq-cpt-supp} if,  whenever $\Omega$ is an exterior domain, namely
$\Omega\supset M\setminus B_R(o)$ for some $R>0$, and $u\geq 0$ is a semiclassical solution
of \eqref{ineq-cpt-supp} in $\Omega$ with the property that $u(x)\to 0$ as
$r(x)\to +\infty$ that $u$ vanishes identically  outside a compact set.

We are grateful to M. Rigoli for pointing  that the validity of the compact
support principle depends essentially on the properties of the operator and
of the function $\lambda$, and that the effect of the geometry is once again encoded by
the behavior of $\Delta_f r(x)$.

Indeed, the proof of Theorem 1.1 in \cite{puriser-jde}, may be adapted with minor
changes to obtain the following

\begin{theorem}
\label{thm-cpt-supp}
Let $M_f$ be a complete weighted manifold, and let $\lambda$ be a
continuous function on $[0,+\infty)$ which vanishes in $0$ and is non-decreasing and
strictly positive on some interval $(0,\delta)$, $\delta>0$. Set $\Lambda (t) =
\int _0^t\lambda (s)ds$. If
\begin{equation}
\label{cpt-supp-1}
\frac {1}{\sqrt{\Lambda (s)}} \in L^1(0+),
\end{equation}
and there exists $C>0$ such that the differential inequality
\begin{equation}
\label{cpt-supp-2}
\Delta_f r(x) \geq -C^2>-\infty
\end{equation}
holds weakly on $M$, then the compact support principle holds for the differential
inequality \eqref{ineq-cpt-supp}.
\end{theorem}
As an immediate consequence, as noted in \cite{puriser-jde}, if \eqref{cpt-supp-1} holds,
$\Delta_f= \Delta$,  and $M$ is a Cartan-Hadamard manifold, then the compact support principle
holds for \eqref{ineq-cpt-supp}.

\subsection{Feller property on Ricci solitons} \label{subsection_Feller&Solitons}

In this section we are going to prove that Ricci solitons are Feller  with
respect to the weighted Laplacian $\Delta_{f}$ and with respect to the
ordinary Laplacian $\Delta$.

%

The fact that for every gradient Ricci soliton the weighted Laplacian
$\Delta_{f}$ is Feller is then a consequence of the results of the previous
section and of the following estimates for the gradient of
the potential function which have been obtained by Zhang,
\cite{Zhang-PAMS} (see also H.-D. Cao and D. Zhou, \cite{CaoZhou-JDG}). These
in turn depend on lower estimates for the scalar curvature of Ricci solitons
obtained by Zhang, and by a different method in \cite{PRimS-MathZ}. We
are going to briefly describe the latter approach.

We recall from \cite{PRimS-MathZ} the following
\textquotedblleft a-priori\textquotedblright\ estimate for weak solutions of
semi-linear elliptic inequalities under volume assumptions (see also
\cite{PRS-gafa}, Theorem B).

\begin{theorem}
\label{th_apriori}Let $\left(  M,\left\langle ,\right\rangle ,e^{-f}%
d\mathrm{vol}\right)  $ be a complete, weighted manifold. Let $a\left(
x\right)  ,$ $b\left(  x\right)  \in C^{0}\left(  M\right)  $, set
$a_{-}\left(  x\right)  =\max\left\{  -a\left(  x\right)  ,0\right\}  $ and
assume that
\[
\sup_{M}a_{-}\left(  x\right)  <+\infty
\]
and
\[
b\left(  x\right)  \geq\frac{1}{Q\left(  r\left(  x\right)  \right)  }\text{
on }M,
\]
for some positive, non-decreasing function $Q\left(  t\right)  $ such that
$Q\left(  t\right)  =o\left(  t^{2}\right)  $, as $t\rightarrow+\infty$.
Assume furthermore that, for some $H>0$,%
\[
\frac{a_{-}\left(  x\right)  }{b\left(  x\right)  }\leq H\text{, on }M.
\]
Let $u\in Lip_{loc}\left(  M\right)  $ be a non-negative solution of%
\begin{equation}
\Delta_{f}u\geq a\left(  x\right)  u+b\left(  x\right)  u^{\sigma}\text{,}
\label{a-priori3}%
\end{equation}
weakly on $\left(  M,e^{-f}d\mathrm{vol}\right)  $, with $\sigma>1$. If%
\begin{equation}
\liminf_{r\rightarrow+\infty}\frac{Q\left(  r\right)  \log\mathrm{vol}%
_{f}\left(  B_{r}\right)  }{r^{2}}<+\infty, \label{a-priori4}%
\end{equation}
then%
\[
u\left(  x\right)  \leq H^{\frac{1}{\sigma-1}}\text{, on }M.
\]

\end{theorem}

Using the volume estimate \eqref{qian_volestimate} we deduce the validity of
the next

\begin{corollary}
\label{cor_scalarcurvestimate}Let $\left(  M,\left\langle ,\right\rangle
,\nabla f\right)  $ be a complete Ricci soliton. Then the scalar curvature $S$
of $M$ satisfies the lower estimate
\begin{equation}
\label{scal_estimate}S(x)\geq%
\begin{cases}
0 & \text{if } \lambda\geq0\\
m \lambda & \text{if } \lambda< 0
\end{cases}
\end{equation}

\end{corollary}

\begin{proof}
Indeed,
\[
\Delta_{f}S=\lambda S-\left\vert Ric\right\vert ^{2}.
\]
and since, by the Cauchy-Schwarz inequality $|\mathrm{Ric}|^{2}\geq\displaystyle\frac{1}%
{m}S^{2}$, we have that $S_{-}(x)=\min\{0,-S(x)\}$ is a weak solution of the differential
inequality
\[
\Delta_{f}S_{-}\geq\lambda S_{-}+\frac{1}{m}S_{-}^{2}.
\]
On the other hand, by \eqref{qian_volestimate}, the condition \eqref{a-priori4}
above is satisfied, and an application of Theorem~\ref{th_apriori} with
$a(x)=\lambda$ and $b(x)=1/m$ shows that
\[
S_{-}\leq m\lambda_{-}
\]
and the conclusion follows.
\end{proof}

Using these estimates into the basic equation
\begin{equation}\label{soliton-basic-eq}
S + |\nabla f|^{2} -2\lambda f= C
\end{equation}
and integrating along minimizing geodesics one obtains the estimates for
the potential function and its gradient described in \cite{Zhang-PAMS} and \cite{CaoZhou-JDG}:

\begin{lemma}
\label{f_estimates} Let $\left(  M,\left\langle ,\right\rangle ,\nabla
f\right)  $ be a complete Ricci soliton. Then there exist positive constants $a$ and
$b$ depending only on the soliton such that
\begin{equation}
\label{f-estimate1}|\nabla f|\leq b + |\lambda| d(x,o) \, \text{ and }\,
|f(x)|\leq a + bd(x,o) +\frac{|\lambda|}2 d(x,o)^2.
\end{equation}
\end{lemma}

Using these results we deduce the following

\begin{proposition}
\label{prop_FellerSoliton1} Let $\left(  M,\left\langle ,\right\rangle ,\nabla
f\right)  $ be a complete Ricci soliton. Then $\Delta_{f}$ and $\Delta$ are
both stochastically complete and Feller.
\end{proposition}

\begin{proof}
Since $Ric_{f}=Ric+\mathop{Hess} f= \lambda g$ by definition, and $$|\nabla
f|\leq C_{1}+|\lambda|d(x,o)$$ by the lemma above, the conditions of
Theorem~\ref{qian-feller} are satisfied and the statement concerning the $f$-Laplacian
$\Delta_{f}$ follows.

In fact, as mentioned above, the above conditions imply that, if $p\in M$ and
$\rho(x)=d(x,p)$ then $\Delta_{f} \rho$ satisfies the estimate
\[
\Delta_{f} \rho\leq\frac{m-1} \rho+ C_{1} \rho+ C_{2}(1+d(o,p)).
\]
On the other hand
\[
|\langle\nabla f, \nabla\rho\rangle| \leq|\nabla f|\leq b + |\lambda|
d(o,x))\leq b + |\lambda|(d(+,p) + \rho(x))
\]
so that
\[
\Delta\rho= \Delta_{f} \rho+\langle\nabla f,\nabla\rho\rangle\leq\frac{m-1}
\rho+ [C_{1} +|\lambda|]\rho+ (C_{2}+b +|\lambda|) (1+d(o,p))
\]
and, again by Theorem~\ref{qian-feller} the Laplacian $\Delta$ is Feller.
\end{proof}

The estimates for the potential function described above allow
us to prove Theorem~\ref{th_fullomoriyau}.

\begin{proof}[Proof of Theorem~\ref{th_fullomoriyau}]
Both statements follow applying Theorem~\ref{generalymoriyau} with the
choice $\gamma(x) = r(x)^2$ and using the estimates for $\nabla f$ described
above together with Qian's estimates for $\Delta_f r$ and $\Delta r$ as in the proof of
Proposition~\ref{prop_FellerSoliton1}.
\end{proof}

\subsection{Applications of the Feller property to geometry and PDE's}

As alluded to at the beginning of Section \ref{section_Feller}, and formalized
in Theorem \ref{th_f(u)}, using the Feller property on a
stochastically complete manifold enables one to extend the investigation
of qualitative properties of solution of PDE's to the case where these are
defined only in a neighborhood at infinity. In this section, we will exemplify the use of
this viewpoint in a
number of different geometric and analytic settings. We stress that the needed stochastic completeness
assumptions
are enjoyed by a very rich family of examples. For instance,  according to
Proposition \ref{prop_FellerSoliton1}, a natural and important framework is
represented by Ricci solitons.
In the non-weighted setting,
we have the class of complete manifolds such that $Ric \geq -G(r)$,
where $G(r)>0$ is a smooth increasing function satisfying $1/\sqrt G \not\in L^1(+\infty)$; \cite{Hsu-AnnProb}. Another
admissible category for the application of Theorem \ref{th_f(u)} is given by Cartan-Hadamard manifolds
with at most quadratic exponential volume growth. Actually, the (radial) sectional
curvature assumption can be considerably relaxed as explained in Theorem 5.9
of \cite{ps-feller}.

\subsubsection{Isometric immersions}

Recall that if a Riemannian manifold $\left(  M,\left\langle ,\right\rangle
\right)  $ is stochastically complete, then the mean curvature $\mathbf{H}$ of
a bounded isometric immersion $f:M\rightarrow\mathbb{B}_{R}\left(  0\right)
\subset\mathbb{R}^{n}$ must satisfy%
\[
\sup_{M}\left\vert \mathbf{H}\right\vert R\geq1.
\]
In particular, a stochastically complete minimal submanifold in Euclidean
space is necessarily unbounded.
 The next result show that this can be extended to the case where we have
 an isometric immersion  of an end $E$ of $M$ with respect to a given compact
 subset $K$ of $M$. We observe that the concepts of stochastic completeness
 and Feller property can be localized  on one end $E$ simply requiring that
 its double $\mathcal{D}(E)$ be stochastically complete or Feller, respectively.
 It is then easy to verify that the maximum principle at infinity holds for every function $f\colon E\to\mathbb{R}$ which is bounded above and does not attain its supremum on $\partial E$. It is easily verified that $M$ is stochastically complete if and only if so are all ends.

 Similarly, the end $E$ is Feller if and only if one (and therefore all) relatively compact domain with smooth boundary $\Omega \supset K$, the minimal positive solution of the boundary value problem
 $$\left\{\begin{array}{llll}\Delta h & =& \lambda\,h & in\,\, M\setminus \overline{\Omega}\\
h&=&1 & on \,\, \partial \Omega
 \end{array}\right.$$ tends to zero as $x\to \infty$ within the given end. Again, it is easy to see that an end $E$ satisfies this condition if and only if its double is Feller, (see \cite{ps-feller} Section 7  for details).

 We are now ready to state our result.  We are grateful to R. Haslhofer and T. Ilmanen for helpful comments related to the formulation of the theorem.

\begin{theorem}
Let $\left(  M,\left\langle ,\right\rangle \right)  $ be a Riemannian manifold and let
$E$ be an end of $M$ with respect to a compact set $K$. Assume that $E$ is stochastically
complete and Feller, and that there exists a bounded isometric immersion $
f:M\rightarrow\mathbb{B}_{R}\left(  0\right)
\subset\mathbb{R}^{n}$. Then
then mean curvature of $f$ satisfies%
\[
\sup_{E}\left\vert \mathbf{H}\right\vert R\geq 1.
\]
\end{theorem}

\begin{proof}
Assume by contradiction that
\begin{equation}
\label{meancurvineq}
\sup_{E} \left\vert \mathbf{H}\right\vert R<1.
\end{equation}
Straightforward computations show that $u=\left\vert f\right\vert ^{2}\geq0$
satisfies
\[
\Delta u\geq c,\text{ on }E,
\]
where we have set%
\[
c=2m\left(  1-\sup_{E}\left\vert \mathbf{H}\right\vert
R\right)  >0,
\]
and $m=dim M$.
If follows that%
\[
\Delta u\geq\lambda u
\]
on $E$ with%
\[
\lambda=\frac{c}{R}\cdot
\]
Now, let $\Omega$ be a bi-collared  relatively compact neighborhood
of $\partial E$ in $\mathcal{D}(E)$. We use the same letters $f$ and $u$ to
denote the obvious extensions of $f$ and $u$ to $\mathcal{D}$. We clearly have
$f(\mathcal{D}(E)\setminus\Omega)\subset \mathbb{B}_{R}(0)$ and
$$
 \begin{array}{lllll}\Delta u &\geq & \lambda\, u & on& \mathcal{D}(E)\setminus \Omega
\end{array}
$$
An application of Theorem \ref{th_f(u)} shows that
$u(x)\to 0$, this is, $f(x)\to 0$ as $x\to +\infty$ in $M$. On the other hand,
since the strict inequality holds in \eqref{meancurvineq},
for $R'>R$ sufficiently close to $R$ we have $\sup_{E}
\left\vert \mathbf{H}\right\vert R'<1$, and
clearly  $f(E)\subset \mathbb{B}_{R'}(0')$ provided $|0'-0|<R'-R$.
Thus we can repeat the argument with $u'(x) = |f(x)-0'|^2$ for which again we have
\[
\Delta u' \geq c
\]
with the same value  $c$, and  then $u'(x)\to 0$, i.e., $f(x)\to 0'\ne 0$,
as $x\to \infty$. This yields the required contradiction and the theorem is proved.
\end{proof}

\subsubsection{Conformal deformations}

Given a Riemannian manifold $\left(  M,\left\langle ,\right\rangle \right)  $
of dimension $m\geq3$ consider the conformally related metric $\overline
{\left\langle ,\right\rangle }=v^{\frac{4}{m-2}}\left\langle ,\right\rangle $
\ where $v>0$ is a smooth function. Thus, the conformality factor $v$ obeys
the Yamabe equation%
\[
c_{m}^{-1}\Delta v-Sv=-\overline{S}v^{\frac{m+2}{m-2}},
\]
where $S$ and $\overline{S}$ denote the scalar curvatures of $\left\langle
,\right\rangle $ and $\overline{\left\langle ,\right\rangle }$, respectively.
Assume that $M$ is stochastically complete and that%
\[
\sup_{M}S\left(  x\right)  \leq S^{\ast},\text{ \thinspace}\inf_{M}%
\overline{S}\left(  x\right)  \geq\overline{S}_{\ast},
\]
for some constants $S^{\ast}\geq0$ and $\overline{S}_{\ast}> 0$. An
application of the weak minimum principle at infinity to the Yamabe equation
\ shows that%
\[
\left(  \frac{S^{\ast}}{\overline{S}_{\ast}}\right)  ^{\frac{m-2}{4}}\geq
v_{\ast}=\inf_{M}v.
\]
In particular, if $S\left(  x\right)  \leq0$ on $M$, then $v_{\ast}=0$.
Actually, since the infimum of $v$ cannot be attained,%
\[
\inf_{M\backslash\Omega}v=0,
\]
for every $\Omega\subset\subset M$. Clearly, to reach these conclusions the
scalar curvature bound must hold on $M$.

\begin{theorem}
Let $\left(  M,\left\langle ,\right\rangle \right)  $ be a stochastically
complete, Feller manifold of dimension $m\geq6$ and scalar curvature
satisfying%
\[
\sup_{M\backslash\Omega}S\left(  x\right)  \leq0,
\]
for some compact domain $\Omega\subset M$. Let $\overline{\left\langle
,\right\rangle }=v^{\frac{4}{m-2}}\left\langle ,\right\rangle $ be a
conformally related metric such that%
\[
\inf_{M}v=v_{\ast}>0.
\]
If the scalar curvature of $(M,\overline{\left\langle ,\right\rangle })$
satisfies%
\[
\inf_{M\backslash\Omega}\overline{S}\left(  x\right)  =\overline{S}_{\ast}>0,
\]
then%
\[
v\left(  x\right)  \rightarrow+\infty\text{, as }x\rightarrow\infty.
\]

\end{theorem}

With respect to the assumption that $S$ is nonnegative at infinity, one may wonder if
it could be made nonnegative everywhere on $M$ with a conformal change of metric. Note however
this in general would require a control on the positive part $S^{+}$ of $S$ in the set $\Omega$, which moreover may depend on the metric itself in a rather implicit way (see,
e.g., Prop 1.2 in \cite{RaRiVe-MathZ}).

\begin{proof}
Just note that the positive, bounded function
$u\left(  x\right)  =v\left(  x\right)  ^{-1}$ satisfies%
\[
c_{m}^{-1}\Delta u\geq-Su+\overline{S}u^{\frac{m-6}{m-2}}\geq\overline
{S}u^{\frac{m-6}{m-2}}.
\]
Since%
\[
0\leq\frac{m-6}{m-2}<1,
\]
Theorem \ref{th_f(u)} yields%
\[
u\left(  x\right)  \rightarrow0\text{, as }x\rightarrow\infty.
\]

\end{proof}

\noindent As an immediate consequence, we obtain the following non-existence result. Note that
this applies, for instance,   to an expanding, gradient Ricci
soliton $M$. Indeed, in this case, the scalar curvature assumption is
compatible with the restriction $\inf_{M}S\leq0$ imposed by the soliton structure.

\begin{corollary}
\label{cor_conformal_feller}
On a stochastically complete and Feller manifold $\left(  M,\left\langle
,\right\rangle \right)  $ of dimension $m\geq6$ and%
\[
\sup_{M\backslash\Omega}S\left(  x\right)  \leq0
\]
one cannot perform a conformal change $\overline{\left\langle ,\right\rangle
}=v^{\frac{4}{m-2}}\left\langle ,\right\rangle $ in such a way that%
\[
0<v_{\ast}\leq v\left(  x\right)  \leq v^{\ast}<+\infty
\]
and%
\[
\inf_{M\backslash\Omega}\overline{S}\left(  x\right)  =\overline{S}_{\ast}>0.
\]

\end{corollary}

\subsubsection{Compact support property of bounded solutions of PDE's}
Recall that a certain PDE satisfies the compact support principle if
a solution, in the exterior of a compact set, which is non-negative and decays at infinity,
must have compact support.
We are going to analyze some situations where the decay assumption can be relaxed.
This has applications to the Yamabe problem.
\begin{theorem}
\label{th_feller+compactsupport}
Let $\left(  M,\left\langle ,\right\rangle \right)  $ be a complete and
stochastically complete, Cartan-Hadamard manifold. Let $u>0$ be a bounded
solution of
\begin{equation}
\label{compactsupport_ineq}
\Delta u\geq\lambda\left(  u\right)  \text{, on }M\backslash\Omega
\end{equation}
for some domain $\Omega\subset\subset M$ and for some non-decreasing function
$\lambda:[0,+\infty)\rightarrow\lbrack0,+\infty)$ satisfying the following
conditions:%
\begin{equation}
\label{Lambda_conditions}
\text{(a) }\lambda\left(  0\right)  =0\text{; (b) }\lambda\left(  t\right)
>0\text{ }\forall t>0\text{; (c) }\liminf_{t\rightarrow0+}\frac{\lambda\left(
t\right)  }{t^{\xi}}>0,
\end{equation}
for some $0\leq\xi<1$. Then $u$ has compact support.
\end{theorem}

\begin{proof}
Recall that a Cartan-Hadamard manifold is Feller (see, \cite{Az-bsmf}, \cite{ps-feller}).
By Theorem \ref{th_f(u)}, we know that $u\left(  x\right)  \rightarrow0$, as
$x\rightarrow\infty$. The conclusion now follows from the compact support principle which,
is  valid under the stated assumptions on $M$ and $\lambda$ (see
\cite{puriser-jde}, Theorem 1.1, and  Theorem~\ref{thm-cpt-supp} above).
\end{proof}

The above theorem can be applied to obtain nonexistence results.
For instance, combining  Theorem~\ref{th_feller+compactsupport} and
Corollary~\ref{cor_conformal_feller}  we get

\begin{corollary}
Let  $\left(  M,\left\langle ,\right\rangle \right)  $ be a stochastically complete
Cartan-Hadamard manifold of dimension $m\geq6$.
Then the metric of $M$ cannot be conformally deformed to a new metric
$\overline{\langle\,,\rangle}= v^2 \langle\,,\rangle$ with $v_*>0$ and scalar curvature
$\overline S$ satisfying $\liminf_{x\to \infty} \overline{S} >0.$
\end{corollary}

Of course for the  conclusion of Theorem~\ref{th_feller+compactsupport} to hold it
suffices that $M$ be stochastically complete, Feller and that the
compact support principle holds for solutions  of
\eqref{compactsupport_ineq}. Theorem~\ref{th_feller+compactsupport} can be therefore
generalized as follows.

\begin{theorem}
Let $\left(  M,\left\langle ,\right\rangle \right)  $ be a complete Riemannian
manifold with a pole $o$ and set $r\left(  x\right)  =d\left(  x,o\right)  $.
Assume that%
\[
Ric\geq-K\left(  r\left(  x\right)  \right)  ,
\]
where $K>0$ is an increasing function satisfying%
\[
\frac{1}{\sqrt{K}}\notin L^{1}\left(  +\infty\right)  .
\]
Assume also that%
\[
Sec_{rad}\leq G\left(  r\left(  x\right)  \right)
\]
where $G$ is a smooth even function such that the
unique solution $g$ of the Cauchy problem
\begin{equation*}
\left\{
\begin{array}
[c]{l}%
g^{\prime\prime}+Gg=0\\
g\left(  0\right)  =0,\text{ }g^{\prime}\left(  0\right)  =1.
\end{array}
\right.  
\end{equation*}
satisfies
\begin{equation*}
\inf\frac {g'}g(t)>-\infty.
\end{equation*}
If $u\geq 0$ is  a bounded solution of
\[
\Delta u\geq\lambda\left(  u\right)  \text{, on }M\backslash\Omega
\]
where $\lambda$ satisfies \eqref{Lambda_conditions},
then $u$ has compact support.
\end{theorem}

In a similar view, using the comparison results established in
Theorem~\ref{th-comparison1} we obtain the following

\begin{theorem}
\label{th_feller+compactsupport1}
Let $M_f $ be a geodesically complete, stochastically complete, weighted manifold with a pole
$o$, and let  by $r(x)$ be the Riemannian distance function from $M$. Suppose
that there exists an integer $n$ and an even function $g:\mathbb{R}\to [0,+\infty )$ such
that $g(0)=0$ $g'(0)=1$ and $g(r)>0$ for $r>0$ such that
\[
\Delta_f r(x) \geq (n-1)\frac{g'}g(r(x)) \quad \text{ for }\, r(x)>>1.
\]
Suppose moreover
\begin{equation*}
\inf\frac {g'}g(t)>-\infty
\end{equation*}
and that either
\begin{equation}
\frac{1}{g^{n-1}\left(  r\right)  }\in L^{1}\left(  +\infty\right)
\label{model1'}%
\end{equation}
or%
\begin{equation}
\text{(i) }\frac{1}{g^{n-1}\left(  r\right)  }\notin L^{1}\left(
+\infty\right)  \text{\quad and\quad(ii) }\frac{\int_{r}^{+\infty}%
g^{n-1}\left(  t\right)  dt}{g^{n-1}\left(  r\right)  }\notin L^{1}\left(
+\infty\right).
\label{model2'}%
\end{equation}
If $u>0$ is  a bounded
solution of
\[
\Delta_f u\geq\lambda\left(  u\right)  \text{, on }M\backslash\Omega
\]
where $\lambda$ satisfies \eqref{Lambda_conditions}, then $u$ has compact support.
\end{theorem}
\begin{proof}
Indeed, since $g$ satisfies the  conditions \eqref{model1'} or \eqref{model2'} it follows from
Theorem~\ref{th-comparison1} (iii) that $M_f$ is Feller. On the other hand, by
Theorem~\ref{thm-cpt-supp}
the compact support principle holds for solutions  of \eqref{compactsupport_ineq}.
Therefore the conclusion follows as in
Theorem~\ref{th_feller+compactsupport}.
\end{proof}


\section{Spectral theory of weighted Laplacians}

\subsection{Basic theory}

In this section we collect some results on the spectral properties of the
$f$-Laplacian. Generally the proofs may be obtained by adapting those valid
for the ordinary Laplacian, and therefore they will be mostly omitted.

The first basic observation is that the $f$-Laplacian is associated to the
$f$-Diriclet form
\begin{equation}
\label{spectral_1}Q_{f}(u) = \int_{M} |\nabla u| ^{2} d\mathrm{vol}_{f},
\end{equation}
originally defined on $C^{\infty}_{c}(M)$. The form $Q_{f}$ is closable and
its closure induces a non-negative self-adjoint operator on $L^{2}%
(d\mathrm{vol}_{f})$, still denoted with $-\Delta_{f}$. The same proof valid
for the usual Laplacian can be adapted to show that $-\Delta_{f}$ is
essentially self-adjoint on $C^{\infty}_{c}(M).$
It is also useful to note that under the unitary transformation $T(u) =
e^{-f/2} u$ of $L^{2}(d\mathrm{vol})$ onto $L^{2}(d\mathrm{vol}_{f})$, the
operator $\Delta_{f}$ is unitarily equivalent to the Schr\"odinger operator
\[
\Delta+ \left(  \frac12 \Delta f - \frac14 |\nabla u|^{2}\right)  ,
\]

More generally, if $\Omega$ is any open set in $M$, we will denote with
$-\Delta_{f}^{\Omega}$ the Friedrichs extension of the operator $-\Delta_{f}$
originally on $C^{\infty}_{c}(\Omega)$. Its domain is given by
\[
Dom(-\Delta_{f}^{\Omega}) = \{u \in H^{1}_{0}(\Omega, d\mathrm{vol}_{f})\,:\,
(\Delta_{f})_{dist} u\in L^{2}(\Omega, d \mathrm{vol}_{f})\}.
\]
The operator $-\Delta_{f}^{\Omega}$ is a positive operator, its spectrum is a
subset of $[0,+\infty)$ and its bottom admits the usual variational
characterization
\[
\lambda_{1}(-\Delta_{f}^ {\Omega}) = \inf\frac{\int_{\Omega}|\nabla u|^{2}
d\mathrm{vol}_{f}}{\int_{\Omega}|u|^{2} d\mathrm{vol}_{f}},
\]
where the infimum is taken over $u\in C^{\infty}_{c}(\Omega)$, or
equivalently, in $H^{1}_{0}(\Omega)$. Also, $\sigma(-\Delta_{f}^{\Omega})$ can
be decomposed into the disjoint union $\sigma_{d}(-\Delta_{f}^\Omega)\cup
\sigma_{ess} (-\Delta_{f}^{\Omega})$, where $\sigma_{d}$ is the set of isolated
eigenvalues of finite multiplicity, called the discrete spectrum, and its complement $\sigma_{ess}$, called the essential spectrum, is the
set of eigenvalues of infinite multiplicity and of accumulations points of the spectrum.

Adapting the arguments valid for the ordinary Laplacian (or using the above
mentioned unitary equivalence with a Schr\"odinger operator, see,
\cite{volpi-masterthesis}) one shows that the following decomposition
principle holds.

\begin{theorem}
\label{decomposition_thm} For every relatively compact domain $\Omega$
\[
\sigma_{ess} (-\Delta_{f}^ M) = \sigma_{ess} (-\Delta_{f}^{M\setminus
\overline{\Omega}}).
\]
In particular,
\[
\inf\sigma_{ess}(-\Delta_{f}^M) = \sup_{\Omega\subset\subset M}\lambda
_{1}(-\Delta_{f}^{M\setminus\overline{\Omega}}).
\]

\end{theorem}

Similarly, one may generalize a result of R. Brooks, \cite{brooks1}, \cite{brooks2}, to obtain
the following upper bound for the infimum of the essential spectrum in terms
of the weighted volume growth of the manifold (see \cite{volpi-masterthesis}).

\begin{theorem}Let $M_{f}$ be a complete weighted manifold.

\begin{enumerate}
\item[(a)] If $\mathrm{vol}_{f}\left(  M\right)  =+\infty$, then%
\[
\limsup_{R\rightarrow+\infty}\frac{\log\mathrm{vol}_{f}\left(  B_{R}\right)  }%
{R}\geq\inf\sigma_{ess}\left(  -\Delta_{f}^{M}\right)  \geq\lambda_{1}\left(
-\Delta_{f}^{M}\right)  \geq0.
\]

\item[(b)] If $\mathrm{vol}_{f}\left(  M\right)  <+\infty$, then%
\[
\limsup_{R\rightarrow+\infty}\frac{-\log\left(  \mathrm{vol}_{f}\left(  M \right)
-\mathrm{vol}_{f}\left(  B_{R}\right)  \right)  }
{R}\geq\inf\sigma
_{ess}\left(  -\Delta_{f}^{M}\right)  \geq\lambda_{1}\left( - \Delta_{f}^{M}\right)
\geq0.
\]

\end{enumerate}
\end{theorem}

The following Barta-type lower estimate for
$\lambda_{1}(-\Delta_{f}^\Omega)$ is a weighted version of a result in \cite{BeMo}.
Its proof is obtained following exactly the arguments in \cite{BeMo}  using
a weighted version of the divergence theorem.

\begin{theorem}
Let $M_{f}$ be a weighted manifold and let
$\Omega\subset M$ be a domain. Then, for every vector field $X$ on $\Omega$%
\[
\lambda_{1}\left(  -\Delta_{f}^\Omega\right)  \geq\inf_\Omega \left\{
\operatorname{div}_{f}\left(  X\right)  -\left\vert X\right\vert ^{2}\right\}
.
\]
The equality holds if $\Omega$ is a compact domain with smooth boundary.
\end{theorem}

A classical consequence is represented by the next
\begin{corollary} Let $M_{f}$ be a weighted manifold and let
$\Omega\subset M$ be a domain. Then, for every domain $\Omega$ and for every
$0<u\in C^{2}\left(  \Omega\right)  $,%
\[
\lambda_{1}\left(  -\Delta_{f}^\Omega\right)  \geq\inf_{\Omega}\left(
-\frac{\Delta_{f}u}{u}\right)  .
\]

\end{corollary}

In particular, recalling Theorem \ref{decomposition_thm}, we deduce
\begin{corollary}\label{cor_essspec}
Let $M_{f}$ be a weighted manifold. Then, for every domain $\Omega
\subset\subset M$ and for every $0<u\in C^{2}\left(  M\backslash\Omega\right)
$, it holds%
\[
\inf\sigma_{ess}\left(  -\Delta_{f}^{M}\right)  \geq\inf_{M\backslash\Omega
}\left(  -\frac{\Delta_{f}u}{u}\right)  .
\]

\end{corollary}

The following version of the classical Cheng eigenvalue comparison was pointed out in \cite{Se}.

\begin{theorem}\label{th_cheng-alberto}
Assume that the complete weighted manifold
$M_{f}$ satisfies%
\[
Ric_{f}\geq-\alpha, \quad\text{and} \quad |\nabla f|\leq \beta^{1/2}
\]
for some $\alpha,\, \beta\geq 0$. Then%
\[
\lambda_{1}\left(-\Delta_{f}^{B_{R}}\right)
\leq\lambda_{1}\left(
-\Delta^{\mathbb{B}_{R}^{m+1}}\right)  ,
\]
where $\mathbb{B}_{R}^{m+1}$ is the ball of radius $R>0$ in the $\left(
m+1\right)  $-dimensional spaceform $\mathbb{M}^{m+1}\left(  (\alpha +\beta)
/m\right)  $ of constant curvature $(\alpha +\beta)/ m$.
\end{theorem}

\subsection{Essential spectrum and stochastic properties}

The purpose of this section is to understand possible connections between the
structure of the spectrum of the Laplacian and the stochastic
properties of noncompact Riemannian manifold, namely, stochastic (in)completeness
and Feller propery. The starting point of the investigation is
represented by the (proof of the) following recent result, \cite{BeJoMo}, which answers in the affirmative a question raised by
S.T. Yau.

\begin{theorem}
\label{th_BJM}Let $M$ be a geodesically complete manifold which admits a
proper minimal immersion $f:M\rightarrow\mathbb{B}_{R}\left(  0\right)  $ into
an open ball $\mathbb{B}_{R}\left(  0\right)  \subset\mathbb{R}^{N}%
\mathbb{\ }$. Then $\sigma_{ess}\left(  -\Delta^M\right)  =\emptyset$.
\end{theorem}

\begin{proof}
According to the decomposition principle,%
\[
\inf\sigma_{ess}\left(  -\Delta^M\right)  =\sup_{\substack{\Omega_{j}\text{
cpt}\\\Omega_{j}\nearrow M}}\lambda_{1}\left(  -\Delta^{M\backslash\Omega_{j}}\right)
,
\]
where, by Barta theorem,%
\[
\lambda_{1}\left( -\Delta^{M\backslash\Omega_{j}}\right)  \geq\sup\inf_{M\backslash
\Omega_{j}}\left(  -\frac{\Delta v}{v}\right)  ,
\]
the supremum being taken with respect to all smooth (say $C^{2}$) functions
$v>0$ on $M\backslash\Omega_{j}$. In particular, choosing%
\[
\Omega_{j}=\left\{  x\in M:\left\vert f\left(  x\right)  \right\vert ^{2}\leq
R^{2}-\frac{1}{j}\right\}  \subset\subset M
\]
and%
\[
v\left(  x\right)  =R^{2}-\left\vert f\left(  x\right)  \right\vert ^{2}>0
\]
gives%
\[
\inf\sigma_{ess}\left( \Delta^M\right)  \geq\lambda_{1}\left( \Delta^{ M\backslash\Omega
_{j}}\right)  \geq\frac{2m}{1/j}\rightarrow+\infty\text{, as }j\rightarrow
+\infty.
\]

\end{proof}

\begin{remark}
\rm{
In the  assumptions of Theorem \ref{th_BJM},  $M$ is stochastically incomplete.
Indeed%
\[
u\left(  x\right)  =\left\vert f\left(  x\right)  \right\vert ^{2}%
\]
is bounded and satisfies%
\[
\Delta u=2m,
\]
thus proving that $u$ violates the weak maximum principle at infinity (in the
terminology of \cite{BeBa}, $u$ is a woymp violating function). By
the maximum principle characterization of stochastic completeness, it follows
that $M$ is stochastically incomplete, as claimed. Moreover, $\ 0\leq u\left(
x\right)  <\sup_{M}u=R^{2}<+\infty$ and $\Omega_{\delta}=\left\{  x\in
M:-\infty<u\left(  x\right)  \leq R^{2}-\delta\right\}  $ defines a compact
exhaustion of $M$, that is, $u:M\rightarrow\lbrack0,R^{2})$ is a proper
function. We are going to prove that these ingredients suffice to conclude the
discreteness of the spectrum, thus establishing an abstract and generalized
version of the main Theorem in \ref{th_BJM}.
}
\end{remark}

\begin{definition}
Say that a function $u:M\rightarrow\left(  -\infty,u^{\ast}\right)  $,
$u^{\ast}<+\infty$, is a bounded exhaustion function if, for every $\delta>0$, the
set%
\[
\Omega_{\delta}=\left\{  x\in M:-\infty<u\left(  x\right)  \leq u^{\ast
}-\delta\right\}
\]
is compact and $\Omega_{\delta}\nearrow M$ as $\delta\rightarrow0$. Note that,
in case $M$ is noncompact, necessarily, $u^{\ast}=\sup_{M}u$.
\end{definition}

\begin{definition}
Let $M_{f}$ be a weighted manifold. A woymp violating function for the
$f$-Laplacian, is a $C^{2}$ function $u:M\rightarrow\mathbb{R}$ satisfying
$\sup_{M}u=u^{\ast}<+\infty$ \ such that, for any sequence $\left\{
x_{k}\right\}  $ along which $u\left(  x_{k}\right)  \rightarrow u^{\ast}$, it
holds $\lim\sup_{k\rightarrow+\infty}\Delta_{f}u\left(  x_{k}\right)  >0$.
\end{definition}

\begin{theorem}
\label{th_ess-spec}Let $M_{f}$ be a ($f$-stochastically incomplete) noncompact
weighted manifold. If $M_{f}$ supports a woymp violating exhaustion function
then
\[\sigma_{ess}\left( -\Delta_{f}^{M}\right)  =\emptyset.\]
\end{theorem}

\begin{proof}
Let $u:M\rightarrow(-\infty,u^{\ast})$, $u^{\ast}=\sup_{M}u<+\infty$, be a
woymp violating exhaustion function. Arguing exactly as above, we consider
$\Omega_{j}=\left\{  x\in M:u\left(  x\right)  \leq u^{\ast}-1/j\right\}
\nearrow M$ and $v\left(  x\right)  =u^{\ast}-u\left(  x\right)  >0$. Note
that, since $u$ is woymp violating,%
\[
\inf_{M\backslash\Omega_{j}}\Delta_{f}u=c_{j}>0,
\]
where, by the obvious monotonicity property of the infimum, the sequence
$c_{j}$ \ is increasing. Therefore,%
\[
\inf\sigma_{ess}\left( -\Delta_{f}^M\right)  \geq\lambda_{1}\left( - \Delta_{f}^{M
\backslash\Omega_{j}}\right)  \geq\inf_{M\backslash\Omega_{j}}\left(
-\frac{\Delta_{f}v}{v}\right)  \geq jc_{j}\rightarrow+\infty\text{, as
}j\rightarrow+\infty.
\]
\end{proof}

At this point, a natural question is to what extent the existence of a woymp
violating exhaustion function characterize the spectrum of the Laplacian of the
underlying manifold.
Beside proper bounded submanifolds with controlled mean curvature,
are there natural examples of manifolds supporting woymp violating exhaustion
functions? Which geometric conditions ensure that such functions exist?
Some examples will help us to focus some important aspects.

\begin{example}
\rm{
\label{ex_StochIncomplete}
Let $M_{g}^{m}\approx\mathbb{R}^{m}$ be a complete, noncompact model manifold
endowed with the metric%
\[
\left\langle ,\right\rangle =dr^{2}+g\left(  r\right)  ^{2}d\theta^{2},
\]
where the smooth function $g:[0,+\infty)\rightarrow\lbrack0,+\infty)$
satisfies%
\[
\left\{
\begin{array}
[c]{ll}%
g\left(  r\right)  >0, & r>0\\
g^{\left(  2k\right)  }\left(  0\right)  =0, & k\in\mathbb{N}\\
g^{\prime}\left(  0\right)  =1. &
\end{array}
\right.
\]
According to a well known characterization $M_{g}^{m}$ \ is stochastically incomplete
if and only if
\begin{equation}
\label{stochincomplete}
\int_{0}^{+\infty}\frac{\int_{0}^{t}g^{m-1}}{g^{m-1}\left(  t\right)
}=u^{\ast}<+\infty.
\end{equation}
Indeed, the function,
\[
u\left(  x\right)  =\int_{0}^{r\left(  x\right)  }\frac{\int_{0}^{t}g^{m-1}%
}{g^{m-1}\left(  t\right)  }dt:M_{g}^{m}\rightarrow\lbrack0,u^{\ast}),
\]
satisfies
\[
\Delta u=1,
\]
and it is bounded if and only \eqref{stochincomplete} holds, in
which case $u$ is in fact a woymp violating exhaustion function, and
$\sigma _{ess}\left( -\Delta^{ M_{g}^{m}}\right)  =\emptyset$.

Thus a model is stochastically incomplete if and only if it admits
a woymp violating exhaustion function. In particular, for a
stochastically incomplete manifold which is a model, the condition
that $\sigma _{ess}\left(  -\Delta^{M_{g}^{m}}\right)  =\emptyset$ is
equivalent with the existence of a woymp violating  exhaustion.
}
\end{example}

\begin{example}
\label{ex-discreteness-vs-stoch-compl}
\rm{
In general it is not true that the stochastic incompleteness of a manifold is
equivalent to the discreteness of the spectrum of its Laplacian. Indeed,
condition $\sigma_{ess}\left(  -\Delta^M\right)  =\emptyset$ is invariant under
bilipschiz diffeomorphisms whereas, according to a result by T. Lyons, \cite{Ly-JDG}, the
stochastic (in)completeness is not. More concretely, consider the Riemannian
product $N=M_{g}^{m}\times\mathbb{R}$ where $M_{g}^{m}$ is stochastically
incomplete model. Then, $N$ is stochastically incomplete but the essential
spectrum of the Laplacian $\Delta_{N}=\Delta_{M}+d^{2}/dt^{2}$ is nonempty.
Note that a natural voymp violating function on $N$ is given by $v\left(
x,t\right)  =u\left(  x\right)  $ where $u\left(  x\right)  $ is defined in
the previous example. Clearly, $v:N\rightarrow\lbrack0,u^{\ast})$ is not an
exhaustion function. This shows that the assumptions of Theorem
\ref{th_ess-spec} are necessary.

One may also investigate whether the condition that the bottom of
the spectrum of exterior balls grows at a specified rate, namely
\begin{equation}
\label{lambda_1 estimate}
\lambda_1 (-\Delta^{M\setminus \overline{B}_R})\geq f(R),
\end{equation}
where $f(R)$ is a monotone nondecreasing and diverges as $R\to +\infty$,
forces stochastic incompleteness. However, without additional
global assumptions even this implication fails.
To see this, we begin observing that, by standard arguments, if $(M,\langle \,,\rangle)$
has a pole $o$ and $|\Delta r|\geq c>0 $ on $M\setminus \overline \Omega$, then
$\lambda_1(-\Delta^{M\setminus \overline \Omega})\geq \displaystyle c^2/4$. Indeed, by continuity
$\Delta r$ has constant sign in every connected component of $M\setminus \overline
\Omega$.
For every  $v\in C^\infty_c (M\setminus \overline \Omega)$ with
support in one of such connected components we have
\[
\int v^2 \leq \frac 1c\int |\Delta r| v^2 =
\pm \frac 2c\int v\langle \nabla v,\nabla r\rangle
\leq \frac 2c \left(\int v^2\right)^{1/2}
\left(\int |\nabla v|^2\right)^{1/2},
\]
and
\begin{equation}
\label{lambda1-estimate-2}
\frac {c^2}4 \int v^2 \leq \int |\nabla v|^2.
\end{equation}
Note that, in particular, it follows from the decomposition principle, that if $|\Delta r|\to
+\infty$ as $x\to \infty$, then $\sigma_{ess}(-\Delta^M) = \emptyset$.

Now, given a function $f$ as above, assume that $g$ satisfies
\[
g(r) = \exp\bigl(-2\int_1^r \sqrt{f(t)}\,dt \bigr)
\]
for sufficiently large $r$. Then the model $M_g^m$  has the property that
\[
\Delta r\leq - 2 \sqrt{f(R)} \quad\text{in the complement of }\, B_R,
\]
and \eqref{lambda_1 estimate} follows. On the other hand, since $g(r)$ is decreasing, $\mathrm{vol}( B_R)$
grows at most linearly, and $M_g^m$ is necessarily stochastically complete.

We note in passing that this shows that the assumption that $M$ be a Cartan-Hadamard manifold plays a fundamental role in Problem 10 of Grigor'yan's
survey \cite{grigoryan-BAMS}.
}
\end{example}

\begin{example}
\rm{
In general a stochastically incomplete manifold with discrete spectrum can
support woymp violating functions both exhaustion and not. Indeed, take
$N=M\times M$ where $M$ supports a positive woymp violating exhaustion
function $u:M\rightarrow\lbrack0,u^{\ast})$ (for instance, $M$ is one of the
bounded minimal surfaces of $\mathbb{R}^{3}$ by Martin-Morales). Then
$\sigma_{ess}\left( -\Delta^ M\right)  =\emptyset$. Moreover, $v\left(  x,y\right)
=u\left(  x\right)  +u\left(  y\right)  $ is woymp violating and exhaustion
whereas $w\left(  x,y\right)  =u\left(  x\right)  $ is woymp violating but not exhaustion.

As a consequence, for a stochastically incomplete manifold $M$, condition
$\sigma_{ess}\left( -\Delta^ M\right)  =\emptyset$ could imply at most that there
exists one woymp violating exhaustion function.
}
\end{example}

In a similar view, since both $\sigma_{ess}$ and the Feller property are only sensitive to
the properties of the manifold off a compact set, it is also natural to investigate to what
extent they are related. However, we are going to see that, without further assumptions, there is
no link between the Feller property and the discreteness of
the spectrum.

\begin{example}
\label{ex_Feller_vs_spectrum}
\rm{
For the sake of simplicity, we restrict ourselves to
the case of the ordinary Laplacian $\Delta$, even though much of the
ensuing discussion could be generalized to the $f$-Laplacian.

%
As noted in the proof of Theorem~\ref{th_feller+compactsupport1},
the model manifold $M^m_g$ is Feller is and only if either
\begin{equation*}
\frac{1}{g^{m-1}\left(  r\right)  }\in L^{1}\left(  +\infty\right)
\end{equation*}
or%
\begin{equation*}
\text{(i) }\frac{1}{g^{m-1}\left(  r\right)  }\notin L^{1}\left(
+\infty\right)  \text{\qquad and\qquad(ii) }\frac{\int_{r}^{+\infty}%
g^{m-1}\left(  t\right)  dt}{g^{m-1}\left(  r\right)  }\notin L^{1}\left(
+\infty\right).
\end{equation*}
Recalling  that $\Delta r = (m-1)\frac {g'}{g}$, we deduce from the discussion in
Example~\ref{ex-discreteness-vs-stoch-compl} that,
if $g(r) =e^{-r^\alpha}$ for $r>>1$,  then $M_g^m$ has discrete spectrum for
every $\alpha >1$, and it is Feller for $\alpha\leq 2$ and
non-Feller for $\alpha>2$. Note that all these manifolds have finite
volume, and so are automatically stochastically complete, showing
that even in the case of models there is no equivalence between
discreteness of the spectrum, and stochastic incompleteness.

It follows that if $M$ is the connected sum of $\mathbb{R}^m$ with a
non-Feller model $M^m_g$, then $M$ is non-Feller. Since the essential spectrum of
$\mathbb{R}^m$ is the entire interval $[0,+\infty)$ an easy argument based
on characteristic sequences supported in the end of $M$  isometric to an exterior domain of
$\mathbb{R}^m$ shows that $\sigma_{ess}(-\Delta^M)=[0,+\infty),$ and $M$ is therefore a non-Feller manifold with
non-empty essential spectrum. Of course, $\mathbb{R}^m$ itself is a trivial example of a Feller manifold
with non-empty essential spectrum.
}
\end{example}


\begin{example}
\rm{
As seen in Example~\ref{ex_StochIncomplete}, a stochastically
incomplete model has discrete spectrum. Since such model has
necessarily infinite volume, by the characterization of the previous
example it is Feller. Small modifications to the arguments described
above show that even for non-Feller, stochastically incomplete
manifolds there is in general no connection with the discreteness of
the spectrum.

Recall that one may extend the definition of stochastic completeness
(incompleteness) to an end of a manifold,  by requiring that the double
of the end be stochastically complete (incomplete).
Then it follows easily from the weak maximum principle that
a manifold is stochastically incomplete if and only so is at least one of
its ends.

The connected sum $M=M_{g_1}^m\# M_{g_2}^m$ of a stochastically incomplete model
$M_{g_1}^m$ with a non-Feller model $M_{g_2}^m$ with discrete
spectrum as described in Example~\ref{ex_Feller_vs_spectrum},
provides an example of a stochastically incomplete, non-Feller
manifold with $\sigma _{ess}\left(- \Delta^ M\right)  =\emptyset$.

On the other hand, the connected sum  $M=M_{g_1}^m\# M_{g_2}^m\#\mathbb{R}^m$
is stochastically incomplete, non-Feller, and $\sigma _{ess}\left(- \Delta^M\right)
=[0,+\infty)$.

Note that all these examples have more than one end, and the case of
manifolds with only one end remains open.
}
\end{example}

\subsection{Spectrum and semilinear PDE's}

In this section we use very easy spectral considerations to deduce information
on\ nonnegative solutions of the differential
inequality%
\begin{equation}
\Delta_{f}u\leq au-bu^{\sigma},\label{ess&eq1}%
\end{equation}
in the exterior of a compact set for some constants $a\geq0,b>0$ and $\sigma>1$. By the strong
minimum principle, $u>0$ unless it is identically zero in each connected
component of every point where it vanishes. If $u$ satisfied the inequality on all of $M_{f}$,
and the weighted manifold was stochastically complete with respect to the
$f$-Laplacian, then a direct application of the weak maximum principle at infinity
 would imply that either $\inf_{M}u=0$ or%
\[
\inf_{M}u\leq\left(  \frac{a}{b}\right)  ^{\frac{1}{\sigma-1}}.
\]
For instance, when applied to the scalar curvature $S\left(  x\right)  $ of a
complete, shrinking Ricci soliton, this procedure gives the estimate%
\[
0\leq\inf_{M}S\leq m\lambda,
\]
where $\lambda$ denotes the soliton constant and $m=\dim M$. Indeed, it is
well known that $S$ satisfies%
\begin{equation}
\frac{1}{2}\Delta_{f}S=\lambda S-\left\vert Ric\right\vert ^{2}\leq\lambda
S-\frac{S^{2}}{m}.\label{soliton-scalareq}%
\end{equation}
In the spirit of the previous sections, we are going to extend these
considerations outside a compact set. This time, however, we use spectral
assumptions instead of stochastic properties.

\begin{proposition}\label{prop_estimate_spec1}
Let $u>0$ be a solution of (\ref{ess&eq1}) in a neighborhood at infinity, for
some constants $a\in \mathbb{R}$, $b>0$ and $\sigma>1$. Then, for every domain
$\Omega\subset\subset M$,%
\[
\inf_{M\backslash\Omega}
u\leq
\left(
\frac{a+ \inf\sigma_{ess}\left(  -\Delta_{f}^M\right)}{b}
\right)
 ^{\frac{1}{\sigma-1} }  .
\]
\end{proposition}
\begin{remark}
{\rm
In particular, if $\inf\sigma_{ess}\left(  -\Delta_{f}^M\right)  =0$, we recover
the above conclusion outside every large compact set. According to Brooks' estimates,
this happens under suitable volume growth assumptions. Moreover, if
$\inf_{M\setminus \Omega} u= (a/b)^{1/(\sigma -1)}$, then $\Delta_f u \leq
0$ on $M\setminus \Omega$ and therefore by the comparison principle $u >(b/a)^{1/(\sigma -1)}$
in the interior of $M\setminus \Omega$. We then conclude that
 $\inf_{M\setminus \Omega} u= (a/b)^{1/(\sigma -1)}$ for
every $\Omega' \supseteq \Omega$, so that the infimum is attained at infinity.
Conversely, if $\inf_{M\backslash \Omega}u>\left(  a/b\right)  ^{1(\sigma-1)}$,
we deduce a gap in the essential spectrum and, therefore, a volume growth
estimate. In particular, this is the case when $a<0$. From another
point of view, when $a<0$ the estimate may be interpreted as a
non-existence result.
}
\end{remark}
\begin{proof}
It is a trivial application of Corollary \ref{cor_essspec}. Indeed%
\[
\inf\sigma_{ess}\left(  -\Delta_{f}^M\right)  \geq\inf_{M\backslash\Omega
}\left(  -\frac{\Delta_{f}u}{u}\right)  \geq-a+b\inf_{M\backslash \Omega%
}u^{\sigma-1}.
\]
\end{proof}
\goodbreak
Going back to the scalar curvature $S(x)$ of a gradient shrinker \ $M_{f}$
with soliton constant $\lambda>0$,
recall that $S\geq0$, the equality holding
at some point if and only if $M_{f}=\mathbb{R}_{f}^{m}$ and $f\left(
x\right)  =A\left\vert x\right\vert ^{2}+\left\langle B,x\right\rangle +c$.
Assume $S>0$, for otherwise there is nothing to prove. Then, for every $R>0$, we obtain the estimates%
\begin{equation}
\label{scal_1}
\inf_{M\backslash B_{R}}\frac{\left\vert Ric\right\vert ^{2}}{S}-\lambda
\leq\inf\sigma_{ess}\left(  -\Delta_{f}^M\right)
\end{equation}
and
\begin{equation}
\label{scal_2}
\inf_{M\backslash B_{R}}S-m\lambda\leq m\inf\sigma_{ess}\left(  -\Delta_{f}^M\right)
.
\end{equation}
In the case of an expanding Ricci soliton, we have $m\lambda \leq \inf_M S \leq 0$. Assume that $\inf_{M\backslash B_{R_o}} S(x)\geq 0$ for some $R_o\geq 0$, so that we must have either $\inf_M S = \inf_{B_{R_o}} S<0$ or $\inf_M S =0$ and the inf is not attained. Exactly as before estimates \eqref{scal_1} and \eqref{scal_2} hold for every $R\geq R_o$ and from the latter we deduce in particular that   $\inf\sigma_{ess}\left(  -\Delta_{f}^M\right)>0$ provided $S>0$ outside a compact set.

In conclusion of this section, we remark that Proposition \ref{prop_estimate_spec1} follows essentially from a suitable application of Barta's
theorem. Further use of this result, but in a slightly different direction,
yields a different kind of information on the solutions at infinity of
(\ref{ess&eq1}).

\begin{proposition}
\label{prop 40}
Let $M_{f}$ be a complete, $m$-dimensional weighted manifold satisfying%
\[
Ric_{f}\geq-\mu,\quad\text{and } \quad|\nabla f|\leq \beta^{1/2}
\]
for some constants $\mu>0$, $\beta \geq  0$. Then, there exists a constant $c=c\left(  m,\mu,\,\beta\right)  >0$
such that the following holds. If $u$ is a solution of
(\ref{ess&eq1}) outside a compact set $K$, for some constants $a \geq 0$, $b>0$ and
$\sigma>1$, then, for every $B_R(x)\in M \setminus K$,%
\[
\left(  \frac{a}{b}+\frac{c}{b}\frac{1+R^{2}}{R^{2}}\right)  ^{\frac{1}{\sigma-1}}\geq\inf
_{B_{R}\left(  x\right)  }u.
\]
In particular, for any fixed $R>0$,%
\[
\limsup_{x\rightarrow\infty}\left\{  \inf_{B_{R}\left(  x\right)  }u\right\}
<+\infty.
\]

\end{proposition}

\begin{proof}
Obviously, the only interesting case is%
\[
\inf_{B_{R}\left(  x\right)  }u>0.
\]
Then,%
\[
\lambda_{1}\left(  -\Delta_{f}^{B_{R}\left(  x\right) } \right)  \geq\inf
_{B_{R}\left(  x\right)  }\left(  -\frac{\Delta_{f}u}{u}\right)  \geq
-a+b\inf_{B_{R}\left(  x\right)  }u^{\sigma-1}
\]
On the other hand, by Theorem \ref{th_cheng-alberto},%
\[
\lambda_{1}\left(  -\Delta_{f}^{B_{R}\left(  x\right) } \right)  \leq\lambda
_{1}\left(  -\Delta^{\mathbb{B}_{R}^{m+1}}\right)  \leq c_{1}\left(
\mu,\beta, m\right)  \left(  1+\frac{1}{R^{2}}\right)
\]
with $\mathbb{B}_{R}^{m+1}\subset\mathbb{H}^{m+1}\left(  (\beta+\mu)/m\right)  .$
Combining these two inequalities completes the proof.
\end{proof}

The above result does not apply as stated to  expanding Ricci solitons, since in this case $|\nabla f|$ cannot be bounded unless the soliton is trivial (see \cite{PRimS-MathZ}). However, by Zhang's estimates, $|\nabla f|$ grows at most linearly in the distance from a reference point (and indeed, it was very recently shown by O. Munteanu and J.P. Wang that its growth is in fact essentially linear).

Assuming that $|\nabla f|\leq C_o r(x)$ in the above argument, it follows that for every fixed $R$ and every $x$ such that  $r(x)\geq 2 R$
\[
\lambda_{1}\left(  -\Delta_{f}^{B_{R}\left(  x\right) }\right)
\leq c_{1}\left(\mu +C^2(R+ r(x))^2\right) (1+ \frac 1 {R^2})
\]
for some constant $C$ depending only on $m$, and we conclude that if $u$ is as in the statement of Proposition~\ref{prop 40} then there exists a constant $c_2$ depending on $m$, $\mu,$ $\sigma$, $C_o$ and $R$  such that
\[
\frac {\inf_{B_R}u}{r(x)^2} \leq  c_2.
\]
This in particular holds for expanding Ricci solitons and compares with Zhang's estimate
\[
\frac{\left\vert S\left(  x\right)  \right\vert }{r\left(  x\right)  ^{2}}\leq
d\left(  m,\lambda\right).
\]
Note that the latter follows applying the estimates of the potential function to
the basic equation (\ref{soliton-basic-eq}), and is therefore specific of the much more rigid geometry imposed by the soliton structure.

\end{document}